\documentclass{amsart}
\usepackage{amsthm,amssymb,amsmath,thmtools}
\usepackage{enumitem}
\usepackage{hyperref}
\newtheorem{theorem}{Theorem}[section]

\newtheorem{proposition}[theorem]{Proposition}

\theoremstyle{definition}
\newtheorem{definition}[theorem]{Definition}

\theoremstyle{remark}

\numberwithin{equation}{section}

\begin{document}

\title[A unique positive entropy solution to a fractional Laplacian]{A note on the existence of a unique positive entropy solution to a fractional Laplacian with singular nonlinearities}

\author{Masoud Bayrami-Aminlouee}

\curraddr{Department of Mathematical Sciences, Sharif University of Technology, Tehran, P.O. Box 11365-9415, Iran.}
\email{masoud.bayrami1990@student.sharif.edu}

\author{Mahmoud Hesaaraki}
\address{Department of Mathematical Sciences, Sharif University of Technology, Tehran, P.O. Box 11365-9415, Iran.}
\email{hesaraki@sharif.edu}

\subjclass[2010]{Primary 35R11, 35J75, 35B09, 35A01}

\keywords{Fractional Laplacian, singular nonlinearity, positive entropy solution, uniqueness}

\begin{abstract} 
In this paper, we study the existence of a positive solution to the following elliptic problem:
$$
\begin{cases}
(-\Delta )^s u = u^{-q} +f(x)h(u)+\mu & \mathrm{in} \,\, \Omega,\\ u>0 & \mathrm{in} \,\, \Omega, \\ u=0 & \mathrm{in} \,\, \big(\mathbb{R}^N \setminus \Omega \big).
\end{cases}
$$
Here $\Omega \subset \mathbb{R}^N$ ($N > 2s$) is an open bounded domain with smooth boundary, $ s \in (0,1)$, and $q \in (0,1)$. For $s \in (0,\frac{1}{2})$, we take advantage of the convexity of $\Omega$. The operator $(-\Delta)^s$ indicates the restricted fractional Laplacian, and $\mu$ is a non-negative bounded Radon measure as a source term. The assumptions on $f$ and $h$ will be precise later. Besides, we will discuss the notion of entropy solution and its uniqueness for some specific measures.
\end{abstract}

\maketitle

\section{Introduction}

This paper is concerned with the existence of a positive solution to the following fractional problem: 
\begin{equation}
\label{Eq1}
\begin{cases}
(-\Delta )^s u = u^{-q} +f(x)h(u)+\mu & \mathrm{in} \,\, \Omega,\\ u>0 & \mathrm{in} \,\, \Omega, \\ u=0 & \mathrm{in} \,\,  \big(\mathbb{R}^N \setminus \Omega \big).
\end{cases}
\end{equation}
Here $\Omega \subset \mathbb{R}^N$ ($N > 2s$) is an open bounded domain with smooth boundary, $ s \in (0,1)$, and $q \in (0,1)$. But, we will use the convexity condition on $\Omega$ for $0 < s < \frac{1}{2}$. Besides, $\mu$ is a non-negative bounded Radon measure as a source term. The assumptions on $f(x)$ and $h(u)$ are as follows:
\begin{enumerate}[leftmargin=*, label=(\Alph*)]
\item \label{Con1} $f \in L^1(\Omega)$ and non-negative.
\item \label{Con2} $h: \mathbb{R}^+ \to \mathbb{R}^+$ is a nonlinear, non-increasing and continuous function such that:
$$ \lim_{s \to 0^+} h(s) \in (0,\infty], \quad \text{and} \quad \lim_{s \to \infty} h(s)=h(\infty)< \infty $$
\end{enumerate}
and also with the following growth conditions near zero and infinity: 
\begin{align*}
& \exists C_1, \underline{K} >0 \,\,\, \text{such that} \,\,\,  h(s) \leq \frac{C_1}{s^{\gamma}} \,\,\, \text{if} \,\,\, s<\underline{K}, \,\,\,
\text{for some} \,\,\, \gamma \in (0,1] \\
& \exists C_2, \overline{K}>0 \,\,\, \text{such that} \,\,\,  h(s)\leq \frac{C_2}{s^{\theta}} \,\,\, \text{if} \,\,\, s > \overline{K}, \,\,\,
\text{for some} \,\,\, \theta > 0. 
\end{align*}
The operator $(-\Delta)^{s} $ stands for the fractional Laplacian which is the non-local generalization of the differential operator $-\Delta u(x) = -\sum_{i=1}^{N} \frac{\partial^2 u}{\partial x_i^2} (x)$, and is given by a singular integral operator in the following way:
\begin{equation}
\label{Eq2}
(-\Delta )^{s} u(x) = C_{N,s} \, \mathrm{P.V.} \int_{\mathbb{R}^N} \frac{u(x)-u(y)}{|x-y|^{N+2s}} \, dy, \qquad u \in \mathcal{S}(\mathbb{R}^N).
\end{equation}
Where $\mathrm{P.V.}$ denotes the Cauchy principal value, $\mathcal{S}(\mathbb{R}^N)$ is the Schwartz space (space of rapidly decreasing functions on $\mathbb{R}^N$) and $C_{N,s}= \frac{4^s \Gamma(\frac{N}{2}+s)}{\pi^{\frac{N}{2}} |\Gamma(-s)|}$, is the normalization constant such that the following identity holds:
\begin{equation}
\label{Eq3}
(-\Delta )^s u = \mathcal{F}^{-1} \big(|\xi|^{2s} \hat{u}(\xi) \big).
\end{equation}
Here $\Gamma$ is the Gamma function and $\mathcal{F}u =\hat{u}$ denotes the Fourier transform of $u$. By restricting the above integral operator to act only on smooth functions that are zero outside $\Omega$, we have the restricted fractional Laplacian $(-\Delta_{|_{\Omega}})^s$, and the zero Dirichlet condition recovers as $ u \equiv 0 $ in $\big(\mathbb{R}^N \setminus \Omega \big)$.

The above two definitions, \eqref{Eq2} and \eqref{Eq3}, along with several other definitions given in \cite{Cite1}, are equivalent. One of them, introduced by Caffarelli and Silvestre \cite{Cite2}, is definition through harmonic extensions. This characterization of $(-\Delta)^s$, is the Dirichlet-to-Neumann map for a local degenerate elliptic PDE in the following way. Let $f \in \mathcal{S}(\mathbb{R}^N)$. If $U=U(x,y): \mathbb{R}^N \times [0,\infty) \to \mathbb{R}$ is the unique solution to
$$
\begin{cases}
\mathrm{div} \big(y^{1-2s} \nabla U \big)=\Delta_x U + \dfrac{1-2s}{y} U_y + U_{yy} =0  & \mathrm{in} \,\, \mathbb{R}^N \times (0,\infty), \\
U(x,0)=f(x)  & \mathrm{on} \,\, \mathbb{R}^N,
\end{cases}
$$
then for any $0<s<1$,
$ (-\Delta )^s f(x) = \frac{4^s\Gamma (s)}{2\Gamma(1-s)} \lim_{y \to 0^+} - y^{1-2s} U_{y}(x,y)$. Caffarelli and Silvestre derived some properties of the fractional Laplace operator from this local argument in the extension problem. For more details about fractional Laplacian and also for the basic properties of the fractional Laplace operator, see \cite{Cite5,Cite3, Cite6, Cite4}.

Problem \eqref{Eq1} arises as a steady-state for the related Heat equation, i.e.
$$
\begin{cases}
u_t + (-\Delta )^s u = u^{-q} +f(x)h(u)+\mu & \mathrm{in} \,\, \Omega \times (0,T), \\
u(x,0)=u_0(x) &   \mathrm{in} \,\, \mathbb{R}^N, \\
u(x,t)>0 &   \mathrm{in} \,\, \Omega \times (0,T), \\ 
u(x,t)=0 &   \mathrm{in} \,\, \big(\mathbb{R}^N \setminus \Omega \big) \times (0,T). 
\end{cases}
$$
The classical Heat equation models many diffusion problems in physics. The Heat equation, as a special case of the diffusive problems, describes how the distribution of some quantity like heat, evolves over time in a solid medium. More general, diffusion problems describe the propagation behavior of the micro-particle mass movement in matter resulting from the random motion of each micro-particle. Recently, studying diffusion problems by replacing the Laplace operator, and its usual variants, by a fractional Laplacian or other similar non-local operators started. For these recent progresses, see the papers \cite{Cite18, Cite19}. Therefore nowadays, studying fractional Heat type equations and their stationary problems are a favorite. 

Besides, the motivation to study problem \eqref{Eq1} comes from the following papers. In paper \cite{Cite13} authors proved the existence of solutions to the following problem:
\begin{equation}
\label{PRLazer}
\begin{cases}
-\Delta  u = f(x) h(u) + \mu   & \mathrm{in} \,\, \Omega \\
u>0   & \mathrm{in} \,\, \Omega \\
u=0   & \mathrm{on} \,\, \partial \Omega,
\end{cases}
\end{equation}
where $\Omega$ is a bounded domain of $\mathbb{R}^N$, $N>2$ and $f$, $h$ and $\mu$ are the same as assumptions in problem \eqref{Eq1} and this paper mainly inspires our problem. Problems as in \eqref{PRLazer} have been extensively studied both for their pure mathematical interest, \cite{Cite40.4042, Cite40.4044, Cite42.3, Cite42.2,Cite42,Cite42.1,Cite40.4041}, and for their relations with some physical phenomena in the theory of pseudoplastic fluids, \cite{Cite40.4040}. Moreover, see \cite{Cite40.4043} for the $p$-Laplacian evolution case of \eqref{PRLazer}.

In \cite{Cite16} Giacomoni, Mukherjee and Sreenadh investigated the existence and stabilization results for the following parabolic equation involving the fractional Laplacian with singular nonlinearity:
$$
\begin{cases}
u_t + (-\Delta )^s u = u^{-q} + f(x,u) &   \mathrm{in} \,\, \Omega \times (0,T), \\
u(x,0)=u_0(x) &   \mathrm{in} \,\, \mathbb{R}^N, \\
u(x,t)>0 &   \mathrm{in} \,\, \Omega \times (0,T), \\ 
u(x,t)=0 &   \mathrm{in} \,\, \big(\mathbb{R}^N \setminus \Omega \big) \times (0,T). 
\end{cases}
$$
Under suitable assumptions on the parameters and datum, they studied the related stationary problem and then using the semi-discretization in time with the implicit Euler method, they proved the existence and uniqueness of the weak solution. It is worth noting that in \cite{Cite20, Cite21}, the authors have shown the same results for the local version of this problem for the general $p$-Laplacian case. 

It is well-known that for $L^1$ or measure data problems, the notion of distributional solution does not ensure uniqueness to the following type of problems:
$$
\begin{cases}
-\mathrm{div} \big(a(x, \nabla u) \big)=\mu    & \Omega, \\
u=0   & \partial \Omega,
\end{cases}
$$
where $\mu$ is a bounded Radon measure or a function in $L^1(\Omega)$. There was an attempt to find some additional conditions on the distributional solutions in order to ensure both existence and uniqueness and some parallel developments achieved, \cite{Cite30}. Although Stampacchia’s definition of solution, \cite{Cite32}, implies uniqueness, it requires stronger conditions on the solution. Namely, his notion of solution uses a larger space of test functions rather than $C^{\infty}_c(\Omega)$. In \cite{Cite24}, the notion of entropy solution introduced for the $L^1$ data and then generalized to some specific measures, \cite{Cite25}. Dall'Aglio, \cite{Cite27}, introduced the notion of SOLA (Solution Obtained as Limit of Approximations) and Lions and Murat, \cite{Cite29,Cite28}, introduced the concept of renormalized solutions. Recently, for the fractional $p$-Laplacian Heat equation Teng, Zhang, and Zhou, \cite{Cite26}, have proved the existence and uniqueness of entropy solution with non-negative $L^1$ data. They have also demonstrated the equivalence of renormalized and entropy solutions. Also, Abdellaoui, Attar, and Bentifour in \cite{Cite36} have studied the existence of an entropy solution to a fractional $p$-Laplacian equation with weight. ‌Besides, see the work \cite{Cite38} in which authors developed an existence, regularity, and potential theory for nonlinear non-local equations involving measure data.
For another approach, we refer the readers to the work \cite{Cite38.1}, where the author studied some integro-differential equations involving measure data by the duality method. Also, see \cite{Cite38.2} for the duality approach to the fractional Laplacian with measure data.

Since our problem \eqref{Eq1} involves a measure term, $\mu$, it is natural to use the notion of entropy solution, which will be defined precisely later in section \ref{Section2}. 

The rest of the paper is organized as follows. In section \ref{Section2}, we will introduce the functional framework. Also, after defining the notions of weak solution and entropy solution to problem \eqref{Eq1}, we will outline our theorem about the existence result. In section \ref{Section3}, we will provide proof of this result. Finally, in section \ref{Section4}, after proving the uniqueness of entropy solution, we will show the existence of it for $L^1$ data, i.e. $\mu \in L^1(\Omega)$.

\section{Functional framework and main result}
\label{Section2}
Let $ 0 < s <1 $ and $ 1 \leq p < \infty $. The classical fractional Sobolev space defines as follows:
$$ W^{s,p}(\mathbb{R}^N) = \Bigg\{ u \in L^p(\mathbb{R}^N) \, : \, \int_{\mathbb{R}^N} \int_{\mathbb{R}^N} \frac{|u(x)-u(y)|^p}{|x-y|^{N+ps}} \, dxdy < \infty \Bigg\} $$
endowed with the Gagliardo norm:
$$ \|u\|_{W^{s,p}(\mathbb{R}^N)} = \|u\|_{L^p(\mathbb{R}^N)} + \Bigg( \int_{\mathbb{R}^N} \int_{\mathbb{R}^N} \frac{|u(x)-u(y)|^p}{|x-y|^{N+ps}} \, dxdy \Bigg)^{\frac{1}{p}}. $$
Also, we define 
$$ X^{s,p}(\Omega) =\Bigg\{ u: \mathbb{R}^N \to \mathbb{R} \,\, \mathrm{measurable}, \, u|_{\Omega} \in L^p(\Omega),\, \iint_{D_{\Omega}} \frac{|u(x)-u(y)|^p}{|x-y|^{N+ps}} \, dxdy < \infty \Bigg\}, $$
where $D_{\Omega}= \mathbb{R}^N \times \mathbb{R}^N \setminus \Omega^{c} \times \Omega^{c} $, with $\Omega^{c}=\mathbb{R}^N \setminus \Omega$ and $\Omega$ is a bounded smooth domain in $\mathbb{R}^N$. This is a Banach space with the following norm:
\begin{equation}
\label{Eq3.1}
\|u\|_{X^{s,p}(\Omega)} = \Bigg( \int_{\Omega} |u|^p \, dx + \iint_{D_{\Omega}} \frac{|u(x)-u(y)|^p}{|x-y|^{N+ps}} \, dxdy \Bigg)^{\frac{1}{p}}. 
\end{equation}
In the case $p=2$, we denote by $X^s(\Omega)$ the space $X^{s,2}(\Omega)$ which is a Hilbert space with the following scalar product:
$$ \langle u,v \rangle_{X^s(\Omega)} = \int_{\Omega} uv \, dx + \iint_{D_{\Omega}} \frac{(u(x)-u(y))(v(x)-v(y))}{|x-y|^{N+2s}} \, dxdy. $$
Moreover, we define $ X_0^{s,p}(\Omega) = \{ u \in X^{s,p}(\Omega) \, : \, u=0 \,\, \mathrm{a.e. \,\, in} \,\,  (\mathbb{R}^N \setminus \Omega ) \}$. Also, we let $X_0^s(\Omega)$ denotes $X_0^{s,2}(\Omega)$. It is easy to see that: 
$$ \Bigg( \int_{\mathbb{R}^N} \int_{\mathbb{R}^N} \frac{|u(x)-u(y)|^p}{|x-y|^{N+ps}} \, dxdy \Bigg)^{\frac{1}{p}} = \Bigg( \iint_{D_{\Omega}} \frac{|u(x)-u(y)|^p}{|x-y|^{N+ps}} \, dxdy \Bigg)^{\frac{1}{p}}, \,\,\, \forall u \in X_0^{s,p}(\Omega). $$
This equality defines another norm equivalent to the norm \eqref{Eq3.1} for $X_0^{s,p}(\Omega)$. We denote this norm by $\|u\|_{X_0^{s,p}(\Omega)}$, i.e.
$$ \|u\|_{X_0^{s,p}(\Omega)} = \Bigg( \iint_{D_{\Omega}} \frac{|u(x)-u(y)|^p}{|x-y|^{N+ps}} \, dxdy \Bigg)^{\frac{1}{p}}. $$
Then there exists a positive constant $C$ such that the following inequalities hold for all $u \in X_0^{s,p}(\Omega)$.
$$ \|u\|_{X_0^{s,p}(\Omega)} \leq \|u\|_{W^{s,p}(\mathbb{R}^N)} \leq C \|u\|_{X_0^{s,p}(\Omega)}. $$
It is worth mentioning that $X_0^{s,p}(\Omega)$ can also be identified by the closure of $C_c^{\infty}(\Omega)$ in $X^{s,p}(\Omega)$. Besides, for the Hilbert space case, we have:
$$ \|u\|_{X_0^s(\Omega)}^2 = 2 C_{N,s}^{-1} \| (-\Delta)^{\frac{s}{2}} u\|_{L^2(\mathbb{R}^N)}^2, $$
where $C_{N,s}$ is the normalization constant in the definition of $(-\Delta)^s$.

For the proofs of the above facts see \cite[subsection 2.2]{Cite31} and \cite{Cite3}.

For $ 0<r<\infty$, the Marcinkiewicz space $M^r(\Omega)$, is the set of all measurable functions $u: \Omega \to \mathbb{R}$, such that there exists $C>0$ with the following condition:
$$  \omega \Big( \Big\{x \in \Omega \, : \, |u(x)| \geq t \Big\} \Big) \leq \frac{C}{t^r}, \qquad \forall t>0. $$
Here $\omega$ denotes the Lebesgue measure on $\mathbb{R}^N$. This space is endowed with the following norm:
$$ \| u\|_{M^r(\Omega)} = \sup_{t>0} \, t \Bigg(\omega \Big( \Big\{x \in \Omega \, : \, |u(x)| \geq t \Big\} \Big) \Bigg)^{\frac{1}{r}}. $$
For every $1 < r < \infty$ and $0<\epsilon \leq r-1$, the following continuous embeddings hold, \cite{Cite24}:
\begin{equation}
\label{Eq3.2}
L^r(\Omega) \hookrightarrow M^r(\Omega) \hookrightarrow L^{r-\epsilon}(\Omega).
\end{equation}
Also the following continuous embeddings will be used in this paper.
\begin{align}
\label{Eq3.3}
& X_0^s(\Omega) \hookrightarrow L^t(\Omega), \qquad \forall t \in[1, 2_s^*], \\
\label{Eq3.350}
& X_0^{s,p}(\Omega) \hookrightarrow L^t(\Omega), \qquad \forall t \in[1, p_s^*],
\end{align}
where $2_s^*=\frac{2N}{N-2s}$ and $p_s^*=\frac{pN}{N-ps}$ are the Sobolev critical exponents. Moreover, these embeddings are compact for $1 \leq t < p_s^*$. See \cite[Theorem 6.5 and Theorem 7.1]{Cite3}.

Since we are dealing with the non-local operator $(-\Delta)^s$, a new class of test functions should be defined precisely instead of the usual one $C_c^{\infty}(\Omega)$, i.e.
$$ \mathcal{T}= \Big \{ \phi:\mathbb{R}^N \to \mathbb{R}  \,\, \big|\,\,  (-\Delta)^s \phi = \varphi, \,\,  \varphi \in  C^{\infty}_c(\Omega), \,\, \phi = 0 \,\,  \mathrm{in} \,\, \mathbb{R}^N \setminus \tilde{\Omega}, \,\,\mathrm{for \,\, some}  \,\, \tilde{\Omega} \Subset \Omega \Big\}. $$
It can be shown that $\mathcal{T} \subset X_0^s(\Omega) \cap L^{\infty}(\Omega)$. Moreover, for every $\phi \in \mathcal{T}$, there exists a constant $\beta \in (0,1)$ such that $\phi \in C^{0,\beta}(\Omega)$. See \cite{Cite37, Cite22, Cite40}.
It is easy to check that for $u\in X_0^s(\Omega)$ and $ \phi \in \mathcal{T}$:
$$
\begin{aligned}
2 C_{N,s}^{-1} \int_{\mathbb{R}^N} u (-\Delta)^{s} \phi \, dx & = 2 C_{N,s}^{-1} \int_{\mathbb{R}^N}  (-\Delta)^{\frac{s}{2}} u (-\Delta)^{\frac{s}{2}} \phi \, dx \\
& = \iint_{D_{\Omega}} \frac{(u(x)-u(y))(\phi(x)-\phi(y))}{|x-y|^{N+2s}} \, dxdy.
\end{aligned}
$$
This equality shows the self-adjointness of $(-\Delta)^s$ in $X_0^s(\Omega)$. Also, one can show that $ (-\Delta)^s : X_0^s(\Omega) \to X^{-s}(\Omega) $ is a continuous strictly monotone operator, where $X^{-s}(\Omega)$ indicates the dual of $X_0^s(\Omega)$. 
\begin{definition}
We say that $u$ is a weak solution to \eqref{Eq1} if:
\begin{itemize}[leftmargin=*]
\item
$u \in L^1_{\mathrm{loc}}(\Omega)$, and for every $K \Subset \Omega$, there exists $C_K>0$ such that $u(x) \geq C_K$ a.e. in $K$ and also $u \equiv 0$ in $\big(\mathbb{R}^N \setminus \Omega \big)$.
\item
Equation \eqref{Eq1} is satisfied in the sense of distributions with the class of test functions $\mathcal{T}$, i.e.
\begin{equation}
\label{Eq3.4}
\int_{\mathbb{R}^N} u (-\Delta)^s \phi \, dx = \int_{\Omega} u^{-q} \phi \, dx + \int_{\Omega}  f h(u) \phi \, dx + \int_{\Omega} \phi \, d\mu, \quad \forall \phi \in \mathcal{T}.
\end{equation}
\end{itemize}
Note that since $\phi$ has compact support in $\Omega$, the first and second terms on the right-hand side of \eqref{Eq3.4} are well-defined by the strict positivity of $u$ on the compact subsets of $\Omega$. Moreover, the last term is well-posed because of test functions belong to $C_c(\Omega)$.
\end{definition}

Concerning the uniqueness, we have another definition to solutions of \eqref{Eq1}. In fact we would like to consider the entropy solution. The motivation of the definition comes from \cite{Cite25}. We will denote
$$
T_k(s) = \begin{cases}
s   & |s| \leq k \\
k \, \mathrm{sign} (s)   & |s| \geq k,
\end{cases}
$$
the usual truncation operator.
\begin{definition}
\label{DEFN}
Let $\mu$ be a measure in $L^1(\Omega)+X^{-s}(\Omega)$. We say that $u$ is an entropy solution to \eqref{Eq1} if:
\begin{itemize}[leftmargin=*]
\item
For every $K \Subset \Omega$, there exists $C_K>0$ such that $u(x) \geq C_K$ in $K$ and also $u \equiv 0$ in $\big(\mathbb{R}^N \setminus \Omega \big)$.
\item
$T_k(u) \in X_0^s(\Omega)$, for every $k$, and $u$ satisfies the following family of inequalities:
\begin{equation}
\label{Eq3.500}
\begin{aligned}
\int_{\{|u-\phi| < k \}} (-\Delta)^{\frac{s}{2}} &  u (-\Delta)^{\frac{s}{2}} (u-\phi) \, dx \leq \int_{\Omega} u^{-q} T_k(u-\phi) \, dx \\
& + \int_{\Omega}  f h(u) T_k(u-\phi) \, dx + \int_{\Omega}T_k(u-\phi) \, d\mu,
\end{aligned}
\end{equation}
for any $k$ and any $\phi \in X_0^s(\Omega) \cap L^{\infty}(\Omega)$. 
\end{itemize}
We will see later that the first and second terms on the right-hand side are well-defined. Moreover, the assumption $\mu \in L^1(\Omega)+X^{-s}(\Omega)$ is for the well-posedness of the measure term, because $T_k(u-\phi) \in X_0^s(\Omega)\cap L^{\infty}(\Omega)$. Notice that $L^1(\Omega)+X^{-s}(\Omega)$ embeds in the dual space of $X_0^s(\Omega)\cap L^{\infty}(\Omega)$.
\end{definition}

\begin{theorem}
\label{Thm1}
Let $ s \in (0,1)$ and $q \in (0,1)$. Also assume that $f$ and $h$ satisfy assumptions \ref{Con1} and \ref{Con2}, respectively. Moreover, $\mu$ is a non-negative bounded Radon measure. Then there is a positive weak solution in $X_0^{s_1,p}(\Omega)$ to problem \eqref{Eq1}, for all $s_1<s$ and for all $p<\frac{N}{N-s}$.
\end{theorem}

\section{Proof of Theorem \ref{Thm1}}
\label{Section3}

First of all, we consider the following auxiliary problem: 
\begin{equation}
\label{Eq4}
\begin{cases}
(-\Delta )^s u =  u^{-q} +  g &   \mathrm{in} \,\, \Omega, \\ u>0 &   \mathrm{in} \,\, \Omega, \\ u=0 &   \mathrm{in} \,\, \big(\mathbb{R}^N \setminus \Omega \big),
\end{cases}
\end{equation}
where $ 0 \leq g \in L^{\infty}(\Omega)$. This problem can be considered as a special case of the Problem $(Q^s)$ in \cite[Theorem 2.9]{Cite16}. For the existence and uniqueness of the solution to problem \eqref{Eq4} we have a modified version of \cite[Theorem 2.9]{Cite16} in the following Proposition.

Before we get into the Proposition, we need to define the set $\mathcal{C}$ as the set of functions $v \in L^{\infty}(\Omega)$ such that there exists positive constants $k_1$ and $k_2$ such that:
\begin{equation}
\label{Eq4.5}
k_1\delta^s(x) \leq v(x) \leq k_2 \delta^s(x).
\end{equation}
Here $ \delta(x)= \mathrm{dist}(x,\partial \Omega)$, $x \in \Omega$, is the distance function from the boundary $\partial \Omega$.
\begin{proposition}
\label{Pro1}
If $ g \in L^{\infty}(\Omega) $, $g \geq 0$, $s \in (0,1)$, and $ q \in (0,1)$, then there exists a unique positive weak energy solution to \eqref{Eq4} in $ X_0^s(\Omega) \cap \mathcal{C} \cap C^{0,s}(\mathbb{R}^N)$. 
\end{proposition}
The notion of the solution to \eqref{Eq4} is as follows. The function $u \in X_0^s(\Omega)$ is a weak energy solution to the above problem if:
\begin{itemize}[leftmargin=*]
\item
For every $K \Subset \Omega$, there exists $C_K>0$ such that $u(x) \geq C_K$ in $K$ and also $u \equiv 0$ in $\big(\mathbb{R}^N \setminus \Omega \big)$.
\item
For every $\phi \in X_0^s(\Omega)$, we have:
\begin{equation}
\label{Eq4.600}
\int_{\mathbb{R}^N} (-\Delta)^{\frac{s}{2}} u (-\Delta)^{\frac{s}{2}} \phi  \, dx =  \int_{\Omega} u^{-q} \phi \, dx + \int_{\Omega} g \phi \, dx.
\end{equation}
\end{itemize}
Note that the first term on the right-hand side of the above equality is well-defined by \eqref{Eq4.5} and applying the H\"older inequality and the fractional Hardy-Sobolev inequality (and convexity of $\Omega$ only for $0<s<\frac{1}{2}$), \cite[Theorem 1.1]{Cite66}. More preciesly, since $u$ behaves like $\delta^s$, we have the following estimate for every $\phi \in X_0^s(\Omega)$.
\begin{equation}
\label{Eq4.700}
\begin{aligned}
\int_{\Omega} u^{-q} \phi \, dx \leq k_1^{-q} \int_{\Omega} \frac{\phi}{\delta^{sq}} \, dx & \leq C_1 \Bigg(\int_{\Omega} \frac{\phi^2}{\delta^{2sq}} \, dx\Bigg)^{\frac{1}{2}} \\ 
& \leq C_2 \|\phi\|_{X_0^{sq}(\Omega)} \leq C_3 \|\phi\|_{X_0^{s}(\Omega)}. 
\end{aligned}
\end{equation}
Here in the last inequality, we used the continuous embedding of $X_0^{s_2}(\Omega)$ into $X_0^{s_1}(\Omega)$, for any $s_1<s_2$. See for example \cite[Proposition 2.1]{Cite3}. 

For general domains with some boundary regularity, the fractional Hardy-Sobolev inequality is proved for $s \in[ \frac{1}{2},1)$. See \cite{Cite66.66, Cite55.555, Cite55.5555}. But in \cite{Cite66}, the authors proved the fractional Hardy-Sobolev inequality for any $s \in (0,1)$, by using the fact that the domain is a convex set and its distance from the boundary is a superharmonic function.

It is worth emphasizing that the uniqueness of the weak energy solution to \eqref{Eq4} follows from the strict monotonicity of the operator $(-\Delta)^{s} u -u^{-q}$, for example see \cite[Lemma 3.1]{Cite16}.

By considering the well-posedness of the first term on the right-hand side of \eqref{Eq4.600}, the well-posedness of the first term on the right-hand side of \eqref{Eq3.500} will be clear after the construction of an entropy solution in section \ref{Section4}. The well-posedness of the second term on the right-hand side of \eqref{Eq3.500} will also be apparent by using assumption \ref{Con2} and the same reasoning.

Now, for every $v \in L^2(\Omega)$, define $\Phi(v)=w$ where $w$ is the solution to the following problem for any fixed $n$ (existence and uniqueness is guaranteed by the above Proposition):
\begin{equation}
\label{Eq5}
\begin{cases}
(-\Delta )^s w = w^{-q} +f_n(x)h_n(|v|+\frac{1}{n})+\mu_n &   \mathrm{in} \,\, \Omega,\\ w>0 &   \mathrm{in} \,\, \Omega, \\ w=0 &   \mathrm{in} \,\,  \big(\mathbb{R}^N \setminus \Omega \big). 
\end{cases}
\end{equation}
Here $ f_n = T_n(f) $ and $ h_n = T_n(h) $ are the truncations at level $n$ and $\{\mu_n\}$ is a sequence of smooth non-negative functions bounded in $L^1(\Omega)$ such that converges to $\mu$ in the weak-star sense of measures, i.e. 
$$ \int_{\Omega} \phi \mu_n \, dx \to \int_{\Omega} \phi  \, d\mu, \qquad \forall \phi \in C_c(\Omega). $$
If we show that $\Phi : L^2(\Omega) \to L^2(\Omega) $ has a fixed point $w_n$, then $w_n \in L^2(\Omega)$ will be the weak solution to the following problem in $X_0^s(\Omega) \cap \mathcal{C} \cap C^{0,s}(\mathbb{R}^N)$.
\begin{equation}
\label{Eq6}
\begin{cases}
(-\Delta )^s w_n = w_n^{-q} +f_n(x)h_n(w_n+\frac{1}{n})+\mu_n &   \mathrm{in} \,\, \Omega,\\ w_n>0 &   \mathrm{in} \,\, \Omega, \\ w_n=0 &   \mathrm{in} \,\,  \big(\mathbb{R}^N \setminus \Omega \big). 
\end{cases}
\end{equation}
For this purpose, we apply the Schauder’s fixed-point theorem. We need to prove that $\Phi$ is continuous, compact and there exists a bounded convex subset of $L^2(\Omega)$ which is invariant under $\Phi$. 

For continuity let $v_k \to v$ in $L^2(\Omega)$. From assumption \ref{Con2} and the dominated convergence theorem it is obvious that for each $n$:
$$ \Big\| (h_n(|v_k|+\frac{1}{n})f_n + \mu_n)-(h_n(|v|+\frac{1}{n})f_n+\mu_n) \Big\|_{L^2(\Omega)} \to 0, \qquad k \to \infty. $$
Now, from the uniqueness of the weak solution to \eqref{Eq4}, we conclude $\Phi(v_k) \to \Phi(v)$. 

For compactness, we argue as follows. For $v \in L^2(\Omega)$, let $w$ be the solution to \eqref{Eq5}. If $\lambda_1^s(\Omega)$ is the first eigenvalue of $(-\Delta)^s$ in $X_0^s(\Omega)$, \cite[Proposition 9]{Cite35}, then we have:
\begin{equation}
\label{Eq7}
\lambda_1^s(\Omega) \int_{\Omega} w^2 \, dx \leq \int_{\mathbb{R}^N} | (-\Delta)^{\frac{s}{2}} w|^2 \, dx.
\end{equation}
Testing \eqref{Eq5} with $\phi=w$, we have:
\begin{equation}
\label{Eq8}
\int_{\mathbb{R}^N} | (-\Delta)^{\frac{s}{2}} w|^2 \, dx = \int_{\Omega} w^{1-q} \, dx + \int_{\Omega} f_n h_n(|v|+\frac{1}{n}) w \, dx + \int_{\Omega} w \mu_n \, dx.
\end{equation}
By the growth condition on $h$, assumption \ref{Con2}, we have:
\begin{align}
\nonumber 
\int_{\Omega} f_n h_n(|v|+\frac{1}{n}) w \, dx & \leq C_1 \int_{\{|v|+\frac{1}{n} < \underline{K}\}} \frac{f_n w}{(|v|+\frac{1}{n})^{\gamma}} + \max_{[\underline{K},\overline{K}]} h \int_{\{\underline{K} \leq |v|+\frac{1}{n} \leq \overline{K}\}} f_n w \, dx \\ \nonumber
& \quad + C_2 \int_{\{|v|+\frac{1}{n} > \overline{K}\}} \frac{f_n w}{(|v|+\frac{1}{n})^{\theta}} \, dx \\ \nonumber
& \leq C_1 n^{1+\gamma} \int_{\{|v|+\frac{1}{n} < \underline{K}\}} |w| \, dx + n \max_{[\underline{K},\overline{K}]} h \int_{\{\underline{K} \leq |v|+\frac{1}{n} \leq \overline{K}\}} |w| \, dx \\ \nonumber
& \quad + C_2 n^{1+\theta} \int_{\{|v|+\frac{1}{n} > \overline{K}\}} |w| \, dx \\ \nonumber
& \leq \Big( C_1 n^{1+\gamma} + n \max_{[\underline{K},\overline{K}]} h + C_2 n^{1+\theta} \Big) \int_{\Omega} |w| \, dx \\
& \leq C_3 \Big( \int_{\Omega} |w|^2 \, dx \Big)^{\frac{1}{2}}, \label{Eq9}
\end{align}
where in the last inequality we have used the H\"older inequality. Once more using the H\"older inequality gives $ \int_{\Omega} \mu_n w \, dx \leq C_4 \Big( \int_{\Omega} |w|^2 \, dx \Big)^{\frac{1}{2}}$ and
$ \int_{\Omega} w^{1-q} \, dx \leq C_5 \Big( \int_{\Omega} |w|^2 \, dx \Big)^{\frac{1-q}{2}}$ for some $C_4>0$ and $C_5>0$. Thus combining recent two inequalities with \eqref{Eq7}, \eqref{Eq8} and \eqref{Eq9} we obtain:
$$ 
\lambda_1^s(\Omega) \int_{\Omega} |w|^2 \, dx \leq C_6 \Big( \int_{\Omega} |w|^2 \, dx \Big)^{\frac{1}{2}} + C_7 \Big( \int_{\Omega} |w|^2 \, dx \Big)^{\frac{1-q}{2}},
$$
which implies that $\Phi(L^2(\Omega))$ is contained in a ball of finite radius in $L^2(\Omega)$. Therefore this ball is invariant under $\Phi$. Moreover, we have $ \int_{\mathbb{R}^N} | (-\Delta)^{\frac{s}{2}} \Phi(v)|^2 \, dx=\int_{\mathbb{R}^N} | (-\Delta)^{\frac{s}{2}} w|^2 \, dx  \leq C_8 $, which means that $\Phi(L^2(\Omega))$ is relatively compact in $L^2(\Omega)$ by the compactness of the embedding \eqref{Eq3.3}.

\begin{proposition}
\label{Pro1.5}
For every $K \Subset \Omega$, there exists $C_K>0$ such that $\{w_n\}$, the solution to \eqref{Eq6}, satisfies $w_n(x) \geq C_K$ a.e. in $K$, for each $n$.
\end{proposition}
\begin{proof}
Let us consider the following problem:
\begin{equation}
\label{Eq9.1}
\begin{cases}
(-\Delta )^s v_n = f_n(x)h_n(v_n+\frac{1}{n}) &   \mathrm{in} \,\, \Omega,\\ v_n>0 &   \mathrm{in} \,\, \Omega, \\ v_n=0 &   \mathrm{in} \,\,  \big(\mathbb{R}^N \setminus \Omega \big). 
\end{cases}
\end{equation}
The existence of the weak solution $v_n$ follows from a similar proof for \eqref{Eq6}. In the same way of the proofs of \cite[Lemma 2.4]{Cite13} and \cite[Lemma 3.2]{Cite22} we can show that $v_n \leq v_{n+1}$ a.e. in $\Omega$ and also for each $K \Subset \Omega$, there exists $C_K>0$ such that $v_1(x) \geq C_K$ a.e. in $K$. 

Now by subtracting the weak formulation of \eqref{Eq9.1} from the weak formulation of \eqref{Eq6} we obtain:
$$ 
\begin{aligned}
\int_{\mathbb{R}^N} (-\Delta)^{\frac{s}{2}} (w_n-v_n) (-\Delta)^{\frac{s}{2}} \phi  \, dx & =\int_{\Omega} f_n \Big[ h_n(w_n+\frac{1}{n})-h_n(v_n+\frac{1}{n}) \Big] \phi \, dx \\
& \quad +  \int_{\Omega} w_n^{-q} \phi \, dx + \int_{\Omega} \mu_n \phi \, dx .
\end{aligned}
$$
Using $\phi=(w_n-v_n)^-$ as a test function and noting that
$$ \int_{\mathbb{R}^N} (-\Delta)^{\frac{s}{2}} (w_n-v_n) (-\Delta)^{\frac{s}{2}} (w_n-v_n)^-  \, dx \leq - \int_{\mathbb{R}^N} |(-\Delta)^{\frac{s}{2}} (w_n-v_n)^-|^2  \, dx $$
see \cite[Proposition 4]{Cite40}, we deduce:
$$ 
\begin{aligned}
-\int_{\mathbb{R}^N} |(-\Delta)^{\frac{s}{2}} (w_n-v_n)^-|^2  \, dx  & \geq  -\int_{\{w_n < v_n\}} f_n \Big[ h_n(w_n+\frac{1}{n})-h_n(v_n+\frac{1}{n}) \Big] (w_n-v_n) \, dx \\
& - \int_{\{w_n < v_n\}} w_n^{-q} (w_n-v_n) \, dx - \int_{\{w_n < v_n\}} \mu_n (w_n-v_n) \, dx.
\end{aligned}
$$
The right-hand side is non-negative since $w_n$ and $\mu_n $ are non-negative functions, and $h_n$ is a non-increasing function. Therefore, we conclude that $(w_n-v_n)^-=0$ or $w_n \geq v_n$ a.e. in $\Omega$. Thus for each $K \Subset \Omega$, there exists $C_K>0$ such that, $w_n \geq v_n \geq v_1 \geq C_{K}>0$ a.e. in $K$.
\end{proof}

\begin{proposition}
\label{Pro2}
For any $k\geq1$, $\{T_k(w_n)\}_{n=1}^{\infty}$ is bounded in $X_0^s(\Omega)$. Moreover, $\|T_k(w_n)\|_{X_0^s(\Omega)}^2=\mathrm{O}(k)$.
\end{proposition} 
\begin{proof}
Taking $\phi=T_k(w_n)$ as a test function in \eqref{Eq6} and invoking \cite[Proposition 3]{Cite40} we get: 
\begin{equation}
\label{Eq10}
\begin{aligned}
\int_{\mathbb{R}^N} & |(-\Delta )^{\frac{s}{2}} T_k(w_n)|^2 \, dx  \leq \int_{\Omega} w_n^{-q} T_k(w_n) \, dx \\
& \quad + \int_{\Omega} f_n(x)h_n(w_n+\frac{1}{n}) T_k(w_n) \, dx + \int_{\Omega}T_k(w_n) \mu_n  \, dx.
\end{aligned}
\end{equation}
For the first term on the right-hand side of \eqref{Eq10} we can write the following estimate:
\begin{align}
\nonumber \int_{\Omega} w_n^{-q} T_k(w_n) \, dx & = \int_{\{w_n \geq k\}} w_n^{-q} k \, dx + \int_{\{w_n <k\}} w_n^{1-q} \, dx  \\ \nonumber
& \leq \int_{\{w_n \geq k\}} k^{1-q} \, dx + \int_{\{w_n <k\}} |T_k(w_n)|^{1-q} \, dx \\ \nonumber 
& = \int_{\{w_n \geq k\}} k^{1-q} \, dx + \int_{\Omega} |T_k(w_n)|^{1-q} \, dx - \int_{_{\{w_n \geq k\}}} k^{1-q} \, dx \\
& = \int_{\Omega} |T_k(w_n)|^{1-q} \, dx. \label{Eq11}
\end{align}
On the other hand, by using H\"older inequality and the embedding \eqref{Eq3.3}:
\begin{equation}
\label{Eq12}
\int_{\Omega} |T_k(w_n)|^{1-q} \, dx \leq C \Bigg( \int_{\Omega} |T_k(w_n)|^{2_s^*} \, dx \Bigg)^{\frac{1-q}{2_s^*}} \leq S^{1-q}C \| T_k(w_n)\|_{X_0^s(\Omega)}^{1-q},
\end{equation}
where $S$ is the best constant in the embedding of $ X_0^s(\Omega) \hookrightarrow L^{2_s^*}(\Omega)$.

For the second term in \eqref{Eq10}, first of all note that 
$$\frac{T_k(w_n)}{(w_n+\frac{1}{n})^{\gamma}} \leq \frac{w_n}{(w_n+\frac{1}{n})^{\gamma}} = \frac{w_n^{\gamma}}{(w_n+\frac{1}{n})^{\gamma} w_n^{\gamma-1}} \leq w_n^{1-\gamma}.$$
Now by using assumption \ref{Con2}, we deduce:
\begin{align}
\nonumber 
\int_{\Omega} f_n h_n(w_n+\frac{1}{n}) T_k(w_n) \, dx & \leq C_1 \int_{\{w_n+\frac{1}{n} < \underline{K}\}} \frac{f_n T_k(w_n)}{(w_n+\frac{1}{n})^{\gamma}} \\ \nonumber 
& \quad + \max_{[\underline{K},\overline{K}]} h \int_{\{\underline{K} \leq w_n+\frac{1}{n} \leq \overline{K}\}} f_n T_k(w_n) \, dx \\ \nonumber
& \quad + C_2 \int_{\{w_n+\frac{1}{n} > \overline{K}\}} \frac{f_n T_k(w_n)}{(w_n+\frac{1}{n})^{\theta}} \, dx \\ \nonumber
& \leq C_1 \underline{K}^{1-\gamma} \int_{\{w_n+\frac{1}{n} < \underline{K}\}} f \, dx \\ \nonumber 
& \quad + k \max_{[\underline{K},\overline{K}]} h \int_{\{\underline{K} \leq w_n+\frac{1}{n} \leq \overline{K}\}} f \, dx \\ \nonumber 
& \quad + \frac{k C_2}{\overline{K}^{\theta}}  \int_{\{w_n+\frac{1}{n} > \overline{K}\}} f \, dx \\
& \leq \Big(C_1 \underline{K}^{1-\gamma} + k \max_{[\underline{K},\overline{K}]} h + \frac{k C_2}{\overline{K}^{\theta}} \Big) \|f\|_{L^1} = C_4+kC_3. \label{Eq13}
\end{align}
Also for the last term:
\begin{equation}
\label{Eq14}
\int_{\Omega}T_k(w_n)  \mu_n \, dx \leq k C_5.
\end{equation}
Thus from \eqref{Eq10}, \eqref{Eq11}, \eqref{Eq12}, \eqref{Eq13} and \eqref{Eq14} we obtain:
\begin{equation}
\label{Eq15}
\int_{\mathbb{R}^N} |(-\Delta )^{\frac{s}{2}} T_k(w_n)|^2 \, dx \leq S^{1-q}C \| T_k(w_n)\|_{X_0^s(\Omega)}^{1-q} + C_4+kC_3 +  k C_5. 
\end{equation}
Now \eqref{Eq15} gives $ \| T_k(w_n)\|_{X_0^s(\Omega)}^2 \leq C_6 \| T_k(w_n)\|_{X_0^s(\Omega)}^{1-q} +  C_4+kC_7 $, which implies the boundedness of $\{T_k(w_n)\}$ in $X_0^s(\Omega)$. This completes the proof of Proposition.
\end{proof}
Now, we have the following Proposition in the spirit of \cite[Theorem 23]{Cite40} and \cite[Theorem 4.10]{Cite37}.
\begin{proposition}
\label{Pro3}
$\{w_n\}_{n=1}^{\infty}$ is bounded in $L^p(\Omega)$, for all $p < \frac{N}{N-2s}$ and bounded in $X_0^{s_1,p}(\Omega)$, for all $s_1<s$ and for all $p < \frac{N}{N-s}$.
\end{proposition} 
\begin{proof}
Once more from the embedding \eqref{Eq3.3} and the previous Proposition we derive:
\begin{equation}
\label{Eq16}
\Bigg( \int_{\Omega} | T_k(w_n)|^{2_s^*} \, dx \Bigg)^{\frac{2}{2_s^*}} \leq S^2 \int_{\mathbb{R}^N} |(-\Delta )^{\frac{s}{2}} T_k(w_n)|^2 \, dx \leq kC.
\end{equation}

Now, consider the set $\{ x \in \Omega \, : \, \big|(-\Delta )^{\frac{s}{2}} w_n(x) \big| \geq t \}$ which the following estimate holds for it:
$$
\begin{aligned}
& \Big\{  x \in \Omega \, : \, \big|(-\Delta )^{\frac{s}{2}} w_n(x) \big| \geq t, \, w_n(x)<k \Big\} \cup \Big\{  x \in \Omega \, : \,  \big|(-\Delta )^{\frac{s}{2}} w_n(x) \big| \geq t, \, w_n(x) \geq k \Big\} \\
& \subset \Big\{  x \in \Omega \, : \,  \big|(-\Delta )^{\frac{s}{2}} w_n(x) \big| \geq t, \, w_n(x)<k  \Big\} \cup \Big\{  x \in \Omega \, : \,  w_n(x) \geq k \Big\} \subset \Omega.
\end{aligned}
$$
Then using the subadditivity property of Lebesgue measure $\omega$ we have:
\begin{equation}
\label{Eq17}
\begin{aligned}
\omega \Big( \Big\{ x \in \Omega \, : \, \big|(-\Delta )^{\frac{s}{2}} w_n(x) \big| \geq t \Big\} \Big) & \leq \omega \Big( \Big\{  x \in \Omega \, : \,  \big|(-\Delta )^{\frac{s}{2}} w_n(x) \big| \geq t, \, w_n(x)<k  \Big\} \Big) \\
& \quad + \omega \Big( \Big\{ x \in \Omega \, : \, w_n(x) \geq k \Big\} \Big).
\end{aligned}
\end{equation}
Now, by using the Chebyshev's inequality and the previous Proposition:
\begin{align}
 \nonumber \omega \Big( \Big\{ x \in \Omega \, : \, \big|(-\Delta )^{\frac{s}{2}} w_n(x) \big| \geq t, \, w_n(x)<k  \Big\} \Big) & \leq \frac{1}{t^2} \int_{\mathbb{R}^N} | (-\Delta )^{\frac{s}{2}} T_k(w_n)|^2 \, dx \\
& \leq \frac{kC}{t^2}, \qquad \forall k\geq 1.  \label{Eq18}
\end{align}
Also from \eqref{Eq16} it is obvious that $k^2 \omega (\{x \in \Omega \, : \, w_n(x) \geq k\})^{\frac{2}{2_s^*}} \leq kC$, or:
\begin{equation}
\label{Eq19}
\omega \Big( \Big\{x \in \Omega \, : \, w_n(x) \geq k \Big\} \Big) \leq \frac{C}{k^{\frac{N}{N-2s}}}, \qquad \forall k\geq 1. 
\end{equation}
Therefore $\{w_n\}$ is bounded in the Marcinkiewicz space $M^{\frac{N}{N-2s}}(\Omega)$ and by the continuous embedding \eqref{Eq3.2}, $w_n$ is bounded in $L^p(\Omega)$, for all $p < \frac{N}{N-2s}$. Also, from \eqref{Eq17}, \eqref{Eq18} and \eqref{Eq19} we have:
$$ \omega \Big( \Big\{x \in \Omega \, : \, \big|(-\Delta )^{\frac{s}{2}} w_n(x) \big| \geq t \Big\} \Big) \leq \frac{kC}{t^2} + \frac{C}{k^{\frac{N}{N-2s}}}, \qquad \forall k\geq 1.$$
Choosing $k=t^{\frac{N-2s}{N-s}}$ gives:
$$  \omega \Big( \Big\{x \in \Omega \, : \, \big|(-\Delta )^{\frac{s}{2}} w_n(x) \big| \geq t \Big\} \Big) \leq \frac{C}{t^{\frac{N}{N-2s}}} + \frac{C}{t^{\frac{N}{N-s}}} \leq \frac{2C}{t^{\frac{N}{N-s}}}, \qquad \forall t\geq 1.  $$
Therefore $\{ (-\Delta )^{\frac{s}{2}} w_n \}$ is bounded in $M^{\frac{N}{N-s}}(\Omega)$ and again the embedding \eqref{Eq3.2} implies the boundedness of $\{ (-\Delta )^{\frac{s}{2}} w_n\}$ in $L^p(\Omega)$, for all $p < \frac{N}{N-s}$. Now, by invoking \cite[Theorem 5 (C) in chapter 5]{Cite41} we get the boundedness of $\{w_n\}$ in $X_0^{s_1,p}(\Omega)$, for all $s_1<s$ and for all $p<\frac{N}{N-s}$.
\end{proof}

Now we are ready to prove the Theorem \ref{Thm1}.

\begin{proof}[Proof of Theorem \ref{Thm1}]
There exists $ u \in X_0^{s_1,p}(\Omega)$ such that up to a subsequence $ w_n \to u$ weakly in $X_0^{s_1,p}(\Omega)$, for all $s_1<s$ and for all $p < \frac{N}{N-s}$. This implies:
$$ \lim_{n \to \infty} \int_{\mathbb{R}^N} (-\Delta )^{\frac{s}{2}} w_n (-\Delta )^{\frac{s}{2}} \phi \, dx = \int_{\mathbb{R}^N} u (-\Delta)^s \phi \, dx, \qquad \forall \phi \in \mathcal{T}.$$
Also, using the embedding \eqref{Eq3.350}, up to a subsequence we may assume that:
\begin{itemize}[leftmargin=*]
\item
$ w_n \to u $ in $L^r(\Omega)$, for any $r \in [1,p_{s_1}^*)$, where $p_{s_1}^*=\frac{pN}{N-p{s_1}}$.
\item
$w_n(x) \to u(x)$ pointwise a.e. in $\Omega$.
\item
There exists $d \in L^r(\Omega)$, for any $r \in [1,p_{s_1}^*)$, such that $|w_n(x)| \leq d(x)$, a.e. in $\Omega$ for all $n$. 
\end{itemize} 
Now for every fixed $\phi \in \mathcal{T}$, by the dominated convergence theorem, we could pass to the limit and obtain:
$$
\begin{aligned}
& \int_{\Omega} w_n^{-q} \phi \, dx \to  \int_{\Omega} u^{-q} \phi \, dx \\
& \int_{\Omega} f_n h_n(w_n+\frac{1}{n}) \phi \, dx \to \int_{\Omega} f h(u) \phi \, dx.
\end{aligned}
$$
Therefore, $u \in X_0^{s_1,p}(\Omega)$, for all $s_1<s$ and for all $p < \frac{N}{N-s}$, is a distributional solution to \eqref{Eq1}. This means that:
$$
\int_{\mathbb{R}^N} u (-\Delta)^s \phi \, dx = \int_{\Omega} u^{-q} \phi \, dx + \int_{\Omega}  f h(u) \phi \, dx + \int_{\Omega}\phi   \, d\mu, \quad \forall \phi \in \mathcal{T}.
$$
Since for every $K \Subset \Omega$, there exists $C_K>0$ such that $w_n(x) \geq C_K$ a.e. in $K$ and also $w_n \equiv 0$ in $\big(\mathbb{R}^N \setminus \Omega \big)$ and because of the pointwise convergence, i.e. $w_n(x) \to u(x)$ a.e. in $\Omega$, thus $u$ is a weak solution to \eqref{Eq1}.
\end{proof} 

\section{Discussion the notion of entropy solution and its uniqueness}
\label{Section4}
As we mentioned it earlier in Introduction, for PDE's concerning measure data problems, the notion of distributional solution does not ensure uniqueness. For this reason, we want to construct an entropy solution (see Definition \ref{DEFN}) and investigate its uniqueness. At first, we show the uniqueness. We will follow the idea of \cite[Section 5]{Cite24}.

Let $u$ and $v$ be two entropy solutions. Testing $u$ with $\phi=T_r(v)$ and $v$ with $T_r(u)$ in the weak formulation of entropy inequalities, we have:
\begin{equation}
\label{Eq20}
\begin{aligned}
& \int_{\{|u-T_r(v)|<k\}} (-\Delta)^{\frac{s}{2}} u (-\Delta)^{\frac{s}{2}} (u-T_r(v)) \, dx \leq \int_{\Omega} u^{-q} T_k(u-T_r(v)) \, dx \\
& \qquad \qquad + \int_{\Omega}  f h(u) T_k(u-T_r(v)) \, dx + \int_{\Omega}T_k(u-T_r(v))  \, d\mu
\end{aligned}
\end{equation}
and
\begin{equation}
\label{Eq21}
\begin{aligned}
& \int_{\{|v-T_r(u)|<k\}} (-\Delta)^{\frac{s}{2}} v (-\Delta)^{\frac{s}{2}} (v-T_r(u)) \, dx \leq \int_{\Omega} v^{-q} T_k(v-T_r(u)) \, dx \\
& \qquad \qquad + \int_{\Omega}  f h(v) T_k(v-T_r(u)) \, dx + \int_{\Omega} T_k(v-T_r(u))  \, d\mu. 
\end{aligned}
\end{equation}
Adding up the left-hand sides of \eqref{Eq20} and \eqref{Eq21} and restricting them to the set
$$ A_0^r=\{ x \in \Omega \,:\, |u-v|<k, \, |u|<r, \, |v|<r \}, $$
we obtain:
\begin{equation}
\label{Eq22}
I_0:= \int_{A_0^r} |(-\Delta)^{\frac{s}{2}} (u-v)|^2 \, dx.
\end{equation}
Also, summing the right-hand sides of \eqref{Eq20} and \eqref{Eq21} when restricted to $A_0^r$ gives:
$$ J_0:=\int_{A_0^r}(u^{-q}-v^{-q}) (u-v) \, dx + \int_{A_0^r} f(x) (h(u)-h(v))(u-v)  \, dx, $$
which is obviously non-positive, i.e. $J_0 \leq 0$. Therefore:
\begin{equation}
\label{Eq23}
I_0=\int_{A_0^r} |(-\Delta)^{\frac{s}{2}} (u-v)|^2 \, dx \leq 0.
\end{equation}
Now, consider the set $ A_1^r=\{ x \in \Omega \,:\, |u-T_r(v)|<k, \, |v| \geq r \}$. 
When restricted to $A_1^r$, the inequality \eqref{Eq20} becomes as follows:
\begin{equation}
\label{Eq24}
\begin{aligned}
\int_{A_1^r} |(-\Delta)^{\frac{s}{2}} u|^2 \, dx  & \leq \int_{A_1^r} u^{-q} (u-r) \, dx \\
& \quad + \int_{A_1^r} f h(u) (u-r) \, dx + \int_{A_1^r}  (u-r)  \, d\mu.
\end{aligned}
\end{equation}
Finally on the set $ A_2^r=\{ x \in \Omega \,:\, |u-T_r(v)|<k, \, |v| < r, \, |u| \geq r \}$, the inequality \eqref{Eq20} is as follows:
\begin{equation}
\label{Eq25}
\begin{aligned}
\int_{A_2^r} (-\Delta)^{\frac{s}{2}} u (-\Delta)^{\frac{s}{2}} (u-v) \, dx  &  \leq  \int_{A_2^r} u^{-q} (u-v) \, dx \\
&  \quad + \int_{A_2^r}  f h(u)(u-v) \, dx + \int_{A_2^r} (u-v)  \, d\mu. 
\end{aligned}
\end{equation}
Similarly, we can estimate \eqref{Eq21} on the sets $B_1^r=\{x \in \Omega \,:\, |v-T_r(u)|<k, \, |u| \geq r \}$ and $B_2^r=\{ x \in \Omega \,:\, |v-T_r(u)|<k, \, |u| < r, \, |v| \geq r  \}$ and find that:
\begin{equation}
\label{Eq26}
\begin{aligned}
\int_{B_1^r} |(-\Delta)^{\frac{s}{2}} v|^2 \, dx  & \leq \int_{B_1^r} v^{-q} (v-r) \, dx \\
& \quad + \int_{B_1^r} f h(v) (v-r) \, dx + \int_{B_1^r}  (v-r)  \, d\mu
\end{aligned}
\end{equation}
and
\begin{equation}
\label{Eq27}
\begin{aligned}
\int_{B_2^r} (-\Delta)^{\frac{s}{2}} v (-\Delta)^{\frac{s}{2}} (v-u) \, dx &  \leq  \int_{B_2^r} v^{-q} (v-u) \, dx \\
& \quad + \int_{B_2^r} f h(v)(v-u) \, dx + \int_{B_2^r}  (v-u)  \, d\mu.
\end{aligned}
\end{equation}

Notice that the right-hand sides of \eqref{Eq24}, \eqref{Eq25}, \eqref{Eq26}, and \eqref{Eq27} goes to zero as $r \to \infty$. Then by combining \eqref{Eq23}, \eqref{Eq24}, \eqref{Eq25}, \eqref{Eq26}, and \eqref{Eq27} we deduce:
$$ I_0=\int_{A_0^r} |(-\Delta)^{\frac{s}{2}} (u-v)|^2 \, dx \leq \mathrm{o}(r), \qquad r \to \infty. $$
Since $A_0^r$ goes to $\{ x\in \Omega \,:\, |u-v|<k \}$, as $r \to \infty$, we obtain that:
$$  \int_{|u-v|<k} |(-\Delta)^{\frac{s}{2}} (u-v)|^2 \, dx \leq 0, $$
or
$$  \int_{\mathbb{R}^N} |(-\Delta)^{\frac{s}{2}} T_k(u-v)|^2 \, dx \leq 0, \qquad \forall k. $$
Therefore, $T_k(u-v) \equiv 0$, for all $k$, and the uniqueness is proved.

For constructing an entropy solution, we assume that $\mu \in L^1(\Omega)$. The reason is that for a non-negative bounded Radon measure $\mu$, generally it is not possible to approximate it with an increasing sequence of $L^{\infty}(\Omega)$ functions, \cite{Cite42}. But for the case $\mu \in L^1(\Omega)$, this can always be done by the usual truncation technique. In the following argument, the increasing sequence of approximations for $\mu$  will ensure the existence of an increasing sequence of solutions to the following approximating problems:
\begin{equation}
\label{Eq30}
\begin{cases}
(- \Delta)^s u_n = u_n^{-q}+ f_n h_n(u_n+\frac{1}{n})+ T_n(\mu) &   \mathrm{in} \,\, \Omega,\\ u_n>0 &   \mathrm{in} \,\, \Omega, \\ u_n=0 &   \mathrm{in} \,\,  \big(\mathbb{R}^N \setminus \Omega \big).
\end{cases}
\end{equation}
Outline of the construction is as follows. 

In the same way of section \ref{Section3}, it is possible to show that $\{ T_k(u_n-\phi)\}_{n=1}^{\infty}$ is a bounded sequence in $X_0^s(\Omega)$ for each fixed $k$ and each fixed $\phi \in X_0^s(\Omega) \cap L^{\infty}(\Omega)$. Also, $\{ T_k(u_n-\phi)\}_{n=1}^{\infty}$ is an increasing sequence of non-negative functions by the strict monotonicity of the operator $(- \Delta)^s u -u^{-q}$ and the increasing behavior of $h_n(u_n+\frac{1}{n}) f_n + T_n(\mu)$. Therefore, up to a subsequence $T_k(u_n-\phi) \to T_k(u-\phi)$ weakly in $X_0^s(\Omega)$ as $n \to \infty$, where $u$ is the weak solution to \eqref{Eq1} with $\mu \in L^1(\Omega)$. Once more the strict monotonicity of $(- \Delta)^s $ implies that $ T_k(u_n-\phi) \to T_k(u-\phi)$ strongly in $X_0^s(\Omega)$ (see for example \cite[Lemma 2.18]{Cite37} for this compactness result). Now, using $T_k(u_n-\phi)$ as a test function in \eqref{Eq30}, and invoking the similar estimates as in \eqref{Eq4.700}, we may pass to the limit by the Vitali convergence theorem and find an entropy solution even with the equalities instead of the inequalities in Definition \ref{DEFN}, i.e. \eqref{Eq3.500}.


\begin{thebibliography}{99}

\bibitem{Cite1} %1
Kwa\'snicki~M. Ten equivalent definitions of the fractional Laplace operator. Fract. Calc. Appl. Anal. 2017;20:7--51. 

\bibitem{Cite2} %2
Caffarelli~L, Silvestre~L. An extension problem related to the fractional Laplacian. Commun. Part. Diff. Eq. 2007;32:1245--1260.

\bibitem{Cite3} %3
Di Nezza~E, Palatucci~G, E. Valdinoci. Hitchhiker's guide to the fractional Sobolev spaces. Bull. Sci. Math. 2012;136:521--573. 

\bibitem{Cite4} %4
Stinga~P. R. User’s guide to the fractional Laplacian and the method of semigroups. Fractional Differential Equations. 2019;235--266. 

\bibitem{Cite5} %5
Lischke~A, Pang~G, Gulian~M, et al. What is the fractional Laplacian?. arXiv. 2018;1801.09767.

\bibitem{Cite6} %6
Silvestre~L. Regularity of the obstacle problem for a fractional power of the Laplace operator. Commun. Pure Appl. Math. 2007;60:67--112.

\bibitem{Cite18} %7
V\'azquez~J. L. Nonlinear diffusion with fractional Laplacian operators. In Nonlinear partial differential equations, Abel Symp. 7 (2012) 271--298.

\bibitem{Cite19} %8
V\'azquez~J. L. Recent progress in the theory of nonlinear diffusion with fractional Laplacian operators. arXiv. 2014;1401.3640.

\bibitem{Cite13} %9
Panda~A, Ghosh~S, Choudhuri~D. Elliptic Partial Differential Equation Involving a Singularity and a Radon Measure. J. Indian Math. Soc. (N.S.). 2019;86:95--117.

\bibitem{Cite42} %10
Boccardo~L, Orsina~L. Semilinear elliptic equations with singular nonlinearities. Calc. Var. Partial Differential Equations. 2010;37:363--380.

\bibitem{Cite42.1} %11
Orsina~L, Petitta~F. A Lazer-McKenna type problem with measures. Differential Integral Equations. 2016;29:19--36.

\bibitem{Cite42.2} %12
Oliva~F, Petitta~F. Finite and infinite energy solutions of singular elliptic problems: Existence and uniqueness. J. Differential Equations. 2018;264:311--340.

\bibitem{Cite42.3} %13
Oliva~F, Petitta~F. On singular elliptic equations with measure sources. ESAIM Control Optim. Calc. Var. 2016;22:289--308.

\bibitem{Cite40.4041} %14
Crandall~M. G, Rabinowitz~P. H, Tartar~L. On a Dirichlet problem with a singular nonlinearity. Commun. Part. Diff. Eq. 1977;2:193--222.

\bibitem{Cite40.4042} %15
Lazer~A. C, McKenna~P. J. On a singular nonlinear elliptic boundary-value problem. Proc. Amer. Math. Soc. 1991;111:721--730.

\bibitem{Cite40.4044} %16
Bougherara~B, Giacomoni~J, Hernandez~J. Existence and regularity of weak solutions for singular elliptic problems. Proceedings of the 2014 Madrid Conference on Applied Mathematics in honor of Alfonso Casal. Electron. J. Differ. Equ. Conf. 2015 Jan 1;22:18--30.

\bibitem{Cite40.4040} %17
Nachman~A, Callegari~A. A nonlinear singular boundary value problem in the theory of pseudoplastic fluids. SIAM J. Appl. Math. 1980;38:275--281.

\bibitem{Cite40.4043} %18
Oliva~F, Petitta~F. A nonlinear parabolic problem with singular terms and nonregular data. Nonlinear Anal. 2019 Mar 19.

\bibitem{Cite16} %19
Giacomoni~J, Mukherjee~T, Sreenadh~K. Existence and stabilization results for a singular parabolic equation involving the fractional Laplacian. 2017;arXiv:1709.01906.

\bibitem{Cite20} %20
Brahim~B, Giacomoni~J. Existence of mild solutions for a singular parabolic equation and stabilization. Adv. Nonlinear Anal. 2015;4:123--134.

\bibitem{Cite21} %21
Badra~M, Bal~K, Giacomoni~J. A singular parabolic equation: Existence, stabilization. J. Differential Equations. 2012;252:5042--5075.

\bibitem{Cite30} %22
Dal Maso~G, Murat~F, Orsina~L, et al. Renormalized solutions of elliptic equations with general measure data. Annali della Scuola Normal Superiore di Pisa-Class di Scienze. 1999;28:741--808.

\bibitem{Cite32} %23
Stampacchia~G. The Dirichlet problem for discontinuous coefficient second order elliptic equations. Ann. Inst. Fourier. 1965;15:189--257.

\bibitem{Cite24} %24
B\'enilan~P, Boccardo~L, Gallouët~T, et al. An $L^1$-theory of existence and uniqueness of solutions of nonlinear elliptic equations. Annali della Scuola Normal Superiore di Pisa-Class di Scienze. 1995;22:241--273.

\bibitem{Cite25} %25
Boccardo~L, Gallou\"et~T, Orsina~L. Existence and uniqueness of entropy solutions for nonlinear elliptic equations with measure data. Ann. Inst. H. Poincaré. 1996;13:539--551.

\bibitem{Cite27} %26
Dall'Aglio~A. Approximated solutions of equations with $L^1$ data. Application to the $H$-convergence of quasi-linear parabolic equations. Ann. Mat. Pura Appl. 1996;170:207--240.

\bibitem{Cite28} %27
Lions~P. L, Murat~F. Sur les solutions renormalis\'ees d'\'equations elliptiques non lin\'eaires. to appear. 1965.

\bibitem{Cite29} %28
Murat~F. Soluciones renormalizadas de EDP elipticas no lineales. Preprint 93023. 1993.

\bibitem{Cite26} %29
Teng~K, Chao~Z, Shulin~Z. Renormalized and entropy solutions for the fractional $p$-Laplacian evolution equations. J. Evol. Equ. 2019;19:559--584.

\bibitem{Cite36} %30
Abdellaoui~B, Attar~A, Bentifour~R. On the Fractional $p$-laplacian equations with weight and general datum. Adv. Nonlinear Anal. 2019;8:144--174. 

\bibitem{Cite38} %31
Kuusi~T, Mingione~G, Sire~Y. Nonlocal equations with measure data. Comm. Math. Phys. 2015;337:1317--1368.

\bibitem{Cite38.1} %32
Petitta~F. Some remarks on the duality method for integro-differential equations with measure data. Adv. Nonlinear Stud. 2016;16:115--124.

\bibitem{Cite38.2}  %33
Karlsen~K. H, Petitta~F, Ulusoy~S. A duality approach to the fractional Laplacian with measure data. Publ. Mat. 2011;151--161.

\bibitem{Cite31} %34
Servadei~R, Valdinoci~E. Mountain Pass solutions for non-local elliptic operators. J. Math. Anal. Appl.
2012;389: 887--898.

\bibitem{Cite22} %35
Barrios~B, De Bonis~I, Medina~M, et al. Semilinear problems for the fractional laplacian with a singular nonlinearity. Open Math. 2015;13.

\bibitem{Cite37} %36
Abdellaoui~B, Medina~M, Peral~I, et al. The effect of the Hardy potential in some Calder\'on--Zygmund properties for the fractional Laplacian. J. Differential Equations. 2016;260:8160--8206.

\bibitem{Cite40} %37
Leonori~T, Peral~I, Primo~A, et al. Basic estimates for solutions of a class of nonlocal elliptic and parabolic equations. Discrete Contin. Dyn. Syst. 2015;35:6031--6068.

\bibitem{Cite66} %38
Brasco~L, Cinti~E. On fractional Hardy inequalities in convex sets. arXiv. 2018;1802.02354.

\bibitem{Cite55.555} %39
Dyda~B. A fractional order Hardy inequality. Illinois J. Math.2004;48:575--588.

\bibitem{Cite55.5555} %40
Filippas~S, Moschini~L, Tertikas~A. Sharp Trace Hardy-Sobolev-Maz'ya Inequalities and the Fractional Laplacian. Arch. Ration. Mech. Anal. 2013;208:109--161.

\bibitem{Cite66.66} %41
Abdellaoui~B, Biroud~K, Primo~A. Nonlinear fractional elliptic problem with singular term at the boundary. Complex Var. Elliptic Equ. 2019;64:909--932.

\bibitem{Cite35} %42
Servadei~R, Valdinoci~E. Variational methods for non-local operators of elliptic type. Discrete Contin. Dyn. Syst. 2013;33:2105--2137.

\bibitem{Cite41} %43
Stein~E. M. Singular integrals and differentiability properties of functions. Princeton university press; 1970.

\end{thebibliography}
\end{document}